\documentclass[11pt,article,reqno]{amsart}
\usepackage{amssymb,url}

\theoremstyle{plain}
\newtheorem{theorem}{Theorem}
\newtheorem{lem}{Lemma}
\newtheorem{cor}{Corollary}

\newcommand{\Stab}{\mathrm{Stab}}
\newcommand{\dom}{\mathrm{dom}\,}
\newcommand{\ran}{\mathrm{ran}\,}

\newcommand{\C}{\mathbb{C}}

\newcommand{\N}{\mathbb{N}}
\newcommand{\Z}{\mathbb{Z}}

\newcommand{\D}{\mathbb{D}}
\newcommand{\T}{\mathbb{T}}

\newcommand{\irD}{\mathcal{D}}
\newcommand{\irH}{\mathcal{H}}

\begin{document}

\title{Bilateral weighted shift operators similar to normal operators}

\author{Gy\"orgy P\'al Geh\'er}
\address{Bolyai Institute, University of Szeged, H-6720 Szeged, Aradi v\'ertan\'uk tere 1, Hungary}
\address{MTA-DE "Lend\"ulet" Functional Analysis Research Group, Institute of Mathematics, University of Debrecen, H-4010 Debrecen, P.O. Box 12, Hungary}
\email{gehergy@math.u-szeged.hu; gehergyuri@gmail.com}
\urladdr{\url{http://www.math.u-szeged.hu/~gehergy/}}

\maketitle

\begin{abstract}
We prove that an injective, not necessarily bounded weighted bilateral shift operator on $\ell^2(\Z)$ is similar to a normal operator if and only if it is similar to a scalar multiple of the simple (i.e. unweighted) bilateral shift operator $S$.
\end{abstract}

\maketitle



\section{Introduction}

Let $\irH$ be a complex Hilbert space. The question whether a given linear operator is similar to an operator belonging to a special class is classical. Investigation of similarity to normal operators goes back to the 1940s. One of the first important results is the famous Sz.-Nagy similarity theorem from 1947. Namely, in \cite{SzN} B. Sz.-Nagy proved that a bounded linear operator $T$ is similar to a unitary operator if and only if $T$ has bounded inverse and we have $\sup\{\|T^n\|\colon n\in\Z\}<\infty$. The "only if" part is obvious. On the contrary, the "if" part was extremely surprising. Since then many results have been provided in this direction. We only mention a few of them: in \cite{BN,Kup1,Kup3} linear resolvent tests were provided; and in a recent paper of the author \cite{Ge}, a test concerning the powers of the operator was given. The investigation of operators which are similar to contractions is another classical direction in the theory of similarity problems. Concerning this topic we refer to the following articles: \cite{Foguel,Halmos,Lebow,Pisier}.

The classes of the so-called weighted bilateral, unilateral or backward shift operators are very useful for an operator theorist (\cite{Ni,Sh}). Besides normal operators these are the next natural classes on which conjectures can be tested. Furthermore, weighted shift operators naturally appear in many other areas of Analysis. Thus several properties of weighted shift operators were charaterized in terms of their weight sequences; for instance normality, hyponormality, super and hypercyclicity e.t.c. (see \cite{Ni,Sa1,Sa2,Sh}). In the present short note we are interested in the problem of how we can characterize similarity to normal operators of injective, bilateral weighted shift operators on $\ell^2(\Z)$. We point out that an injective unilateral shift operator cannot be similar to a normal operator since its range is not dense, but it has trivial kernel.

Usually a weighted bilateral shift operator is defined as follows. Let $\underline{w} = \{w_k\}_{k\in\Z}\subseteq\C$ be an arbitrary sequence and $\{e_k\}_{k\in\Z}$ be the usual orthonormal base in $\ell^2(\Z)$. Let $\irD = \{x = \sum_{k\in\Z} \xi_k \cdot e_k \colon \sum_{k\in\Z} |\xi_k|^2 <\infty, \; \sum_{k\in\Z} |\xi_k w_k|^2 <\infty\}$ which is a linear manifold in $\ell^2(\Z)$. We define $S_{\underline{w}} \colon \ell^2(\Z) \to \ell^2(\Z)$ by the following equation:
\[
S_{\underline{w}}\left(\sum_{k\in\Z} \xi_k \cdot e_k\right) = \sum_{k\in\Z} \xi_{k-1}w_{k-1} \cdot e_k \qquad \left(\sum_{k\in\Z} \xi_k \cdot e_k \in \irD\right).
\]
The domain of an operator $T\colon \irH \to \irH$ will be denoted by $\dom T$. Using this terminology, we can write $\irD = \dom S_{\underline{w}}$. It is quite straightforward to see that $S_{\underline{w}}$ is bounded exactly when the weight sequence $\underline{w} = \{w_k\}_{k\in\Z}$ is bounded. In the special case when all weights are equal to 1, we speak about the simple bilateral shift, $S$. It is well-known that $S$ is a unitary operator.

Our result is stated and proven in the next section.

\section{The characterization}

Characterization of normality for bilateral weighted shift operators can be obtained by a simple calculation (see also \cite{Sh}). Namely, $S_{\underline{w}}$ is normal if and only if we have $|w_k| = |w_{k+1}|$ for every $k\in\Z$, which is satisfied exactly when $S_{\underline{w}}$ is unitarily equivalent to $|w_0|\cdot S$. Here we consider the much harder question: which weighted bilateral shift operators are similar to normal operators? Since this is a very natural question, it is quite surprising that we have not found any paper which considers this problem.

Before giving our proof of this result, we shall specify what we mean exactly by the similarity of two (unbounded) operators. We will follow the definition given in \cite[p. 213]{OS}. Namely, two operators $A, B\colon \irH\to\irH$ are said to be similar, if there exists an invertible operator $X\colon \irH\to\irH$ such that $X(\dom A) = \dom B$ and $X A h = B X h$ $(h\in\dom A)$ are fulfilled. Similarity is an equivalence relation (see \cite{OS}). Let us assume that $A$ and $B$ are similar. Are $A^n$ and $B^n$ necessarily similar as well ($n\in\N$)? This is a natural question, which is obvious in case of bounded operators. The following lemma gives the positive answer in general.

\begin{lem}
Suppose that $A$ and $B$ are similar, and $n\in\N$. Then $A^n$ and $B^n$ are also similar.
\end{lem}

\begin{proof}
We have to show that $X A^n h = B^n X h$ $(h\in\dom A^n)$ and $X(\dom A^n) = \dom B^n$ hold. We will use an induction on $n$, thus we assume that for smaller powers the statement has already been proven. First, it is well-known that $\dom A^n\subseteq \dom A^{n-1}\subseteq\irH$. Therefore we can write the following:
\[
X (\dom A^n) = X \left( \left\{ h\in \irH \colon Ah \in \dom A^{n-1} = X^{-1} (\dom B^{n-1}) \right\} \right)
\]
\[
= X \left( \left\{ h\in \dom A^{n-1} \colon XAh = BXh \in \dom B^{n-1} \right\} \right) 
\]
\[
= \left\{ Xh\in X\dom A^{n-1} = \dom B^{n-1} \colon B(Xh) \in \dom B^{n-1} \right\} = \dom B^n.
\]
Second, let $h\in\dom A^n$. Then we have $X A^{n-1} A h = B^{n-1} X A h = B^n X h$, since $Ah \in \dom A^{n-1}$. This completes the proof. \qed
\end{proof}

Now, we are in the position to state and prove our result.

\begin{theorem}\label{bil_thm}
Let $S_{\underline{w}}$ be an injective, not necessarily bounded bilateral weighted shift operator on $\ell^2(\Z)$. Then the following conditions are equivalent:
\begin{itemize}
\item[(a)] $S_{\underline{w}}$ is similar to a normal operator,
\item[(b)] there exists a constant $c>0$ such that $c\cdot S_{\underline{w}}$ is similar to $S$, 
\item[(c)] we can find a constant $c>0$ such that we have
\[ 
\sup \Big\{ c^n\prod_{j=1}^{n} |w_{k+j}| \colon k\in\Z, n\in\N \Big\} < \infty 
\]
and
\[ 
\inf \Big\{ c^n\prod_{j=1}^{n} |w_{k+j}| \colon k\in\Z, n\in\N \Big\} > 0.
\]
\end{itemize}
\end{theorem}

\begin{proof}
First of all, the (c)$\Longrightarrow$(b) part follows immediately from \cite[Theorem 2]{Sh}. Second, the (b)$\Longrightarrow$(a) part is obvious, since $S$ is a normal operator. Thus we only have to deal with the (a)$\Longrightarrow$(c) part. We consider an arbitrary operator $T\colon \irH\to\irH$, and we define the following set:
\[
\Stab(T) = \left\{x\in\bigcap_{n\in\N} \dom T^n \colon \lim_{n\to\infty} \|T^n x\| = 0\right\}
\]
which is clearly a linear manifold. Obviously the zero vector is always an element of $\Stab(T)$. We note that in case when we have $\sup_{n\in\N} \|T^n\| < \infty$, the set $\Stab(T)$ is a subspace (i.e. closed linear manifold) which is hyperinvariant for $T$ (see e.g. \cite{NFBK}). 

Concerning $S_{\underline{w}}$, we claim that either $\Stab(S_{\underline{w}}) = \{0\}$ or it is a dense linear manifold in $\ell^2(\Z)$. Suppose that we have a non-zero vector $x = \sum_{k\in\Z} \xi_k \cdot e_k \in\Stab(S_{\underline{w}})$, then there is an index $k_0\in\Z$ such that $\xi_{k_0} \neq 0$. The condition $x\in\cap_{n=1}^\infty \dom S_{\underline{w}}^n$ is equivalent to the following: 
\begin{equation}\label{powdom_eq}
\sum_{k\in\Z} |\xi_k w_k \dots w_{k+n-1}|^2 < \infty \qquad (\forall\;n\in\N).
\end{equation}
Therefore we get $e_{k_0}\in\cap_{n=1}^\infty \dom S_{\underline{w}}^n$. Since we have
\[
\|S_{\underline{w}}^n e_{k_0}\| \leq \left\|S_{\underline{w}}^n \left(\frac{1}{\xi_{k_0}}\cdot x\right)\right\| = \frac{1}{|\xi_{k_0}|} \cdot \left\|S_{\underline{w}}^n x\right\| \to 0 \quad (n\to\infty),
\]
we conclude $e_{k_0}\in\Stab(S_{\underline{w}})$. Now, let $j\in\N$ be arbitrary. On the one hand, it is quite straightforward that $S_{\underline{w}}^j e_{k_0} \in \Stab(S_{\underline{w}})$, hence we obtain $e_{k_0+j} \in \Stab(S_{\underline{w}})$. On the other hand, by \eqref{powdom_eq} we have $e_{k_0-j}\in \cap_{n=1}^\infty \dom S_{\underline{w}}^n$. Obviously, 
\[
\lim_{n\to\infty} \|S_{\underline{w}}^n e_{k_0-j}\| = 
\lim_{n\to\infty} |w_{k_0-j}\dots w_{k_0-1}|\cdot\|S_{\underline{w}}^{n-j} e_{k_0}\| = 0,
\]
which yields $e_{k_0-j}\in\Stab(S_{\underline{w}})$. Therefore we conclude that every element of the standard orthonormal base lies in $\Stab(S_{\underline{w}})$. Thus $\Stab(S_{\underline{w}})$ has to be dense, which verifies our statement.

Now, we show that any normal operator $N$ satisfies 
\begin{equation}\label{Nstab_eq}
\Stab(N) = \ran E(\D),
\end{equation}
where $\D$ denotes the open unit disk of $\C$, and $E$ is the spectral measure for $N$. Let us consider an arbitrary vector $x = x_1\oplus x_2 \in \cap_{n=1}^\infty \dom N^n$ with $x_1 \in \ran E(\D)$ and $x_2 \in \ran E(\C\setminus\D)$. Then we have $N^n x = (N^n x_1)\oplus(N^n x_2)$. A straightforward application of the function model for normal operators (see e.g. \cite{Co}) gives us $\lim_{n\to\infty}\|N^n x_1\| = 0$ and 
\[
\lim_{n\to\infty}\|N^n x_2\| = \left\{\begin{matrix}
\|x_2\| & \text{if } x_2 \in \ran E(\T),\\
\infty & \text{otherwise},
\end{matrix}\right.
\]
where $\T$ is the unit circle of $\C$. Therefore we obtain that $x\in\Stab(N)$ implies $x\in \ran E(\D)$. On the other hand, if we have $x\in \ran E(\D)$, then we easily obtain $x\in\Stab(N)$, which implies \eqref{Nstab_eq}.

Next, let us suppose that we have $NX h = XS_{\underline{w}} h$ $(h\in\dom S_{\underline{w}})$ with some invertible operator $X\colon \irH\to\irH$ which also satisfies $X(\dom S_{\underline{w}}) = \dom N$. By Lemma 1, we obtain $(c\cdot N)^nX h = X(c\cdot S_{\underline{w}})^n h$ $(h\in\dom S_{\underline{w}}^n, c>0)$ and $X(\dom S_{\underline{w}}^n) = \dom N^n$ for every $n\in\N$. Therefore we have 
\[
\Stab(c\cdot N) = X(\Stab(c \cdot S_{\underline{w}})) = X(\Stab(S_{c \cdot \underline{w}})). 
\]
By the observations made above, we conclude the following:
\begin{equation}\label{Stab_eq}
\Stab(c\cdot N) = \ran E\left(\tfrac{1}{c}\cdot \D\right) \in \left\{\ell^2(\Z),\{0\}\right\} \quad (\forall\; c>0).
\end{equation}
We conclude that we have only two possibilities: either $N = 0$ or $E$ is concentrated on a circle centred at zero. The first one contradicts to the injectivity of $S_{\underline{w}}$. Therefore we have the second case, which means that $c\cdot S_{\underline{w}}$ is similar to a unitary operator with some number $c>0$.

Finally, we use the Sz.-Nagy similarity theorem and the following well-known equations (see \cite[Proposition 2]{Sh}) in order to complete the proof:
\[
\left\|(c\cdot S_{\underline{w}})^n\right\| = \sup \left\{ c^n \cdot \prod_{j=1}^n |w_{k+j}| \colon k\in\Z, n\in\N \right\} \quad (n\in\N)
\]
and
\[
\left\|(c\cdot S_{\underline{w}})^{-n}\right\| = \sup \left\{ \frac{1}{c^n \cdot \prod_{j=1}^n |w_{k+j}|} \colon k\in\Z, n\in\N \right\} \quad (n\in\N).
\]
\end{proof}

We note that the equivalence of (b) and (c) could be proven by a simple application of the Sz.-Nagy theorem as well. In fact, the operator $A$ (which depends on a Banach limit), defined by Sz.-Nagy, is diagonal with respect to the standard orthonormal base in this case, and the unitary operator, to which $c\cdot S_{\underline{w}}$ is similar, is a weighted bilateral shift operator with weights of unit modulus (see \cite{Ge,Ku,NFBK,SzN}).

We close our paper with the following consequence.

\begin{cor}
Let $S_{\underline{w}}$ be an injective weighted bilateral shift operator on $\ell^2(\Z)$ which is similar to a normal operator. Then $S_{\underline{w}}$ is a bounded, cyclic operator, which is not supercyclic. Furthermore, its spectrum is a circle centred at zero.
\end{cor}

\begin{proof}
Since $S$ is cyclic, bounded, and its spectrum is $\T$, we obtain that $S_{\underline{w}}$, which is similar to $\tfrac{1}{c}\cdot S$ with some $c>0$, is also cyclic, bounded, and its spectrum is $\tfrac{1}{c}\cdot\T$. Furthermore, $\tfrac{1}{c}\cdot S$ is not supercyclic by \cite[Theorem 3.1]{Sa2}, which implies that neither is $S_{\underline{w}}$
\end{proof}

\section*{Acknowledgement}
The author was supported by the "Lend\"ulet" Program (LP2012-46/2012) of the Hungarian Academy of Sciences.

\end{document}